\documentclass[a4paper,11pt]{article}

\usepackage[left=3cm, right=3cm]{geometry}
\usepackage[hidelinks]{hyperref}
\usepackage{amsfonts,amssymb,amsmath,amsthm} 
\usepackage[utf8]{inputenc} 
\usepackage[alphabetic, nobysame]{amsrefs}
\usepackage{enumerate}
\usepackage[english]{babel}
\usepackage{authblk}
\usepackage{pstricks,pst-plot,pst-node}
\usepackage[linesnumbered, ruled, vlined]{algorithm2e}
\usepackage{listings}
\usepackage{xcolor}
\usepackage{url}
\usepackage{chngcntr}
\usepackage{graphicx}
\usepackage{filecontents}

\counterwithin{table}{section}

\DefineSimpleKey{bib}{myurl}

\newcommand\myurl[1]{\url{#1}}

\newcommand{\textbx}[1]{{\color{red} #1}}

\BibSpec{webpage}{%
  +{}{\PrintAuthors} {author}
  +{,}{ \textit} {title}
  +{}{ \parenthesize} {date}
  +{,}{ \myurl} {myurl}
  +{,}{ } {note}
  +{.}{ } {transition}
}

\newtheorem{theorem}{Theorem}[section]
\newtheorem*{theorem*}{Theorem}
\newtheorem*{problem*}{Problem}
\newtheorem{lemma}[theorem]{Lemma}

\newtheorem{proposition}[theorem]{Proposition}
\newtheorem{corollary}[theorem]{Corollary}
\newtheorem*{corollary*}{Corollary}
\newtheorem*{conjecture*}{Conjecture}
\newtheorem*{question*}{Question}

\theoremstyle{definition}
\newtheorem{definition}[theorem]{Definition}
\newtheorem*{definition*}{Definition}
\newtheorem{example}[theorem]{Example}
\newtheorem{remark}[theorem]{Remark}

\newcommand{\suchthat}{\;\ifnum\currentgrouptype=16 \middle\fi|\;}

\newcommand{\bigslant}[2]{{\raisebox{.2em}{$#1$}\left/\raisebox{-.2em}{$#2$}\right.}}

\DeclareMathOperator{\Aut}{\mathrm{Aut}}

\DeclareMathOperator{\Z}{\mathbf{Z}}

\title{A lattice in a residually non-Desarguesian $\tilde{A}_2$-building}
 
\author{Nicolas Radu\thanks{F.R.S.-FNRS Research Fellow.}}

\date{October 12, 2016}

\begin{document}
\maketitle

\begin{abstract}
We build a building of type $\tilde{A}_2$ and a discrete group of automorphisms acting simply transitively on its set of vertices. The characteristic feature of this building is that its rank~$2$ residues are isomorphic to the Hughes projective plane of order $9$, which is non-Desarguesian. This solves a problem asked by W. Kantor in 1986, as well as a question asked by J. Howie in 1989.
\end{abstract}

\tableofcontents

\section{Introduction}
An $\tilde{A}_2$-building can be characterized as a simply connected simplicial complex of dimension~$2$ such that all simplicial spheres of radius~$1$ around vertices are isomorphic to the incidence graph of a projective plane. Given an $\tilde{A}_2$-building, we will call these projective planes the \textbf{residue planes} of the building.

The residue planes of an $\tilde{A}_2$-building associated to an algebraic group are always Desarguesian. On the other hand, the existence of locally finite $\tilde{A}_2$-buildings with non-Desarguesian residue planes has been known since 1986. M. Ronan indeed gave in~\cite{Ronan} a general construction affording all possible $\tilde{A}_2$-buildings and from which it is clear that any projective plane can appear in an $\tilde{A}_2$-building. However, this point of view does not provide any information on the automorphism groups of the buildings. A natural problem which was asked by W. Kantor in~\cite{Kantor}*{Page~124} is therefore the following.\footnote{The original statement is ``Construct finite $\tilde{A}_2$-SCABs with non-Desarguesian residues" and is equivalent to ours. See \S\ref{subsection:construction} for the definition of $\tilde{A}_2$-SCABs.}

\begin{problem*}[Kantor, 1986]
Construct an $\tilde{A}_2$-building $\Delta$ with finite non-Desarguesian residue planes and admitting a cocompact lattice, i.e. a group $\Gamma \leq \Aut(\Delta)$ with finite vertex stabilizers and such that $\Gamma \backslash \Delta$ is compact.
\end{problem*}

Independently, J. Howie also asked the next more specific question \cite{Howie}*{Question~6.12}.\footnote{The original question is ``Is there a special presentation with star graph isomorphic to the incidence graph of a finite non-Desarguesian projective plane?". As mentioned in~\cite{Vdovina}, the notion of a special presentation from \cite{Howie} is equivalent to the notion of a triangle presentation from~\cite{CMSZ} (see \S\ref{section:previous} below for the definition of a triangle presentation). The latter is in turn related to $\tilde{A}_2$-buildings, as explained in \S\ref{section:previous}.}

\begin{question*}[Howie, 1989]
Does there exist a group acting simply transitively and by type-rotating automorphisms (as defined in Theorem~\ref{maintheorem} (3) below) on the vertices of an $\tilde{A}_2$-building whose residue planes are non-Desarguesian?
\end{question*}

Some constructions in~\cite{HVM} and~\cite{Barre} yield $\tilde{A}_2$-buildings with exotic residues, but without any lattice (or, as the case may be, without information on its possible existence). In this paper, we solve the above problem and question by proving the following.

\begin{theorem}\label{maintheorem}
There exist an $\tilde{A}_2$-building $\Delta$ and a group $\Gamma \leq \Aut(\Delta)$ satisfying the following properties:
\begin{enumerate}
\item[(1)] All residue planes of $\Delta$ are isomorphic to the Hughes plane of order~$9$.\footnote{The Hughes plane of order $9$ was actually first constructed by O. Veblen and J. Wedderburn in 1907, see~\cite{VW}. This was the first discovered finite non-Desarguesian projective plane, and the role of D. Hughes in~\cite{Hughes} has been to generalize their construction to get an infinite family of finite non-Desarguesian planes (with order $p^{2n}$ for $p$ an odd prime).}
\item[(2)] The group $\Gamma$ acts simply transitively on the set of vertices of $\Delta$.
\item[(3)] Each element of $\Gamma$ is type-rotating, i.e. its induced action $\sigma$ on the set of types of vertices $\{0,1,2\}$ satisfies $\sigma(i) = i+c \bmod 3$ for some $c$.
\item[(4)] The index~$3$ subgroup $\Gamma^0$ of $\Gamma$ consisting of the type-preserving automorphisms is torsion-free.
\item[(5)] The derived subgroup $[\Gamma, \Gamma]$ of $\Gamma$ is perfect and $\bigslant{\Gamma}{[\Gamma, \Gamma]} \cong \mathbf{C}_2 \times \mathbf{C}_3$.
\item[(6)] There exists an infinite family $\{\Delta_0^n\}_n$ of disjoint isomorphic sub-buildings of $\Delta$ whose residue planes are isomorphic to $\mathrm{PG}(2,3)$ and such that each vertex of $\Delta$ is contained in (exactly) one sub-building $\Delta_0^n$. 
\item[(7)] The stabilizer of a vertex in $\Aut(\Delta)$ has order~$96$, i.e. $[\Aut(\Delta) : \Gamma] = 96$. In particular, $\Aut(\Delta)$ equipped with the topology of pointwise convergence is discrete.
\end{enumerate}

\end{theorem}

Moreover, as any unimodular locally compact group acting continuously, properly and cocompactly on an $\tilde{A}_2$-building, $\Gamma$ satisfies Kazhdan's Property (T) (see~\cite{PropertyT}*{Theorem~5.7.7}). Groups with Property (T) are deeply studied in \cite{PropertyT}.

\section*{Acknowledgements}\label{ackref}
I thank Pierre-Emmanuel Caprace for suggesting me this problem and for all his precious advice during the writing of this paper. I am also very grateful to Tim Steger for his interest in this work and for sharing some of his unpublished codes that helped in describing the full automorphism group of the discovered building. I finally want to thank Donald Cartwright and William Kantor for their comments on a previous version of this text, and Alina Vdovina for making me realize that I had answered a question of Howie.

\section{Previous work on the subject}\label{section:previous}

In \cite{CMSZ}, Cartwright, Mantero, Steger and Zappa were interested in groups acting simply transitively on the vertices of an $\tilde{A}_2$-building. We will make great use of their work and give in this section the essential definitions and results.

\subsection{Point-line correspondences and triangle presentations}

For our needs, the most important definition from \cite{CMSZ} is the following.

\begin{definition}\label{definition:trianglepresentation}
Let $P$ and $L$ be the sets of points and lines respectively in a projective plane $\Pi$. A bijection $\lambda \colon P \to L$ is called a \textbf{point-line correspondence} in $\Pi$. A subset $\mathcal{T} \subseteq P^3$ is then called a \textbf{triangle presentation} compatible with $\lambda$ if the two following conditions hold:
\begin{enumerate}
\item For all $x, y \in P$, there exists $z \in P$ such that $(x,y,z) \in \mathcal{T}$ if and only if $y \in \lambda(x)$ in $\Pi$. In this case, $z$ is unique.
\item If $(x,y,z) \in \mathcal{T}$, then $(y,z,x) \in \mathcal{T}$.
\end{enumerate}
\end{definition}

\begin{example}\label{example:order2}
The projective plane $\mathrm{PG}(2,2)$ can be defined by $P = L = \bigslant{\Z}{7\Z}$ with line $x \in L$ being adjacent to the points $x+1$, $x+2$ and $x+4$ in $P$. Consider the point-line correspondence $\lambda\colon P \to L\colon x\in P \mapsto x\in L$ in $\Pi$. Then
\[\mathcal{T} := \{(x,x+1,x+3), (x+1,x+3,x), (x+3,x,x+1) \mid x \in P\}\]
is a triangle presentation compatible with $\lambda$. Indeed, (ii) is obviously satisfied and, for $x,y \in P$, it is apparent that there exists (a unique) $z \in P$ such that $(x,y,z) \in \mathcal{T}$ if and only if $y \in \{x+1,x+2,x+4\}$, which is exactly the set of points on the line $\lambda(x)$.
\end{example}

Now suppose we have an $\tilde{A}_2$-building $\Delta$ and a group $\Gamma \leq \Aut(\Delta)$ acting simply transitively on $V(\Delta)$, where $V(\Delta)$ is the set of vertices of $\Delta$. We make the further assumption that $\Gamma$ only contains \textit{type-rotating} automorphisms, as defined in Theorem~\ref{maintheorem}~(3). In this context, the following theorem shows how one can associate to $\Gamma$ a point-line correspondence and a triangle presentation compatible with it.

\begin{theorem}[Cartwright--Mantero--Steger--Zappa]\label{theorem:presentation}
Let $\Gamma \leq \Aut(\Delta)$ be a group of type-rotating automorphisms of an $\tilde{A}_2$-building which acts simply transitively on $V(\Delta)$. Let $P$ (resp. $L$) be the set of neighbors of type $1$ (resp. $2$) of a fixed vertex $v_0$ of type $0$, and denote by $\Pi$ the residue plane at $v$ (with $P$ and $L$ as sets of points and lines). For each $x \in P$, let $g_x$ be the unique element of $\Gamma$ such that $g_x(v_0) = x$. Then there exist a point-line correspondence $\lambda \colon P \to L$ in $\Pi$ and a triangle presentation $\mathcal{T}$ compatible with $\lambda$ such that $\Gamma$ has the following presentation:
\[\Gamma = \langle \{g_x\}_{x \in P} \mid g_x g_y g_z = 1 \text{ for each } (x,y,z) \in \mathcal{T}\rangle.\]
\end{theorem}

\begin{proof}
See \cite{CMSZ}*{Theorem~3.1}.
\end{proof}

What makes triangle presentations really interesting is the fact that a reciprocal result exists. Given a projective plane $\Pi$, a point-line correspondence $\lambda \colon P \to L$ in $\Pi$ and a triangle presentation compatible with $\lambda$, one can construct an $\tilde{A}_2$-building $\Delta$ locally isomorphic to $\Pi$ and a group acting simply transitively on $V(\Delta)$.

\begin{theorem}[Cartwright--Mantero--Steger--Zappa]\label{theorem:reconstruct}
Let $P$ and $L$ be the sets of points and lines in a projective plane $\Pi$, let $\lambda \colon P \to L$ be a point-line correspondence in $\Pi$ and let $\mathcal{T}$ be a triangle presentation compatible with $\lambda$. Define 
\[\Gamma_{\mathcal{T}} := \langle \{a_x\}_{x \in P} \mid a_x a_y a_z = 1 \text{ for each } (x,y,z) \in \mathcal{T} \rangle,\]
where $\{a_x\}_{x \in P}$ are distinct letters. Then there exists an $\tilde{A}_2$-building $\Delta_{\mathcal{T}}$ whose residue planes are isomorphic to $\Pi$ and such that $\Gamma_{\mathcal{T}}$ acts simply transitively on $V(\Delta_{\mathcal{T}})$, by type-rotating automorphisms.
\end{theorem}

\begin{proof}
See \cite{CMSZ}*{Theorem~3.4}, or \S\ref{subsection:construction} below.
\end{proof}

\begin{example}
From Example~\ref{example:order2} and Theorem~\ref{theorem:reconstruct}, we get an $\tilde{A}_2$-building $\Delta$ whose residue planes are isomorphic to $\mathrm{PG}(2,2)$ and a group acting simply transitively on the set of vertices of $\Delta$. The building $\Delta$ is actually the Bruhat-Tits building associated to $\mathrm{PGL}(3,\mathbf{F}_2(\!(X)\!))$ (see \cite{CMSZ2}*{Section~4} and \cite{CMSZ}*{Theorem~4.1}).
\end{example}

\subsection{Building associated to a triangle presentation}\label{subsection:construction}

In \cite{CMSZ}*{Theorem~3.4}, the authors gave an explicit construction of the $\tilde{A}_2$-building $\Delta_{\mathcal{T}}$ associated to a triangle presentation $\mathcal{T}$ (see Theorem~\ref{theorem:reconstruct} above). In this section we show a geometric way to construct $\Delta_{\mathcal{T}}$ and $\Gamma_\mathcal{T}$. The following discussion can also be seen as an alternative proof of Theorem~\ref{theorem:reconstruct}.

Following~\cite{Kantor}, an \textbf{$\tilde{A}_2$-SCAB} is a connected chamber system of rank~$3$ whose residues of rank~$2$ are generalized $3$-gons (i.e. incidence graphs of projective planes). We will always think of an $\tilde{A}_2$-SCAB as a set of triangles, representing the chambers, glued together so that two chambers are adjacent if and only if they share an edge.

Suppose we are given a point-line correspondence $\lambda \colon P \to L$ in a projective plane $\Pi$ and a triangle presentation $\mathcal{T}$ compatible with $\lambda$. Let us first define a finite $\tilde{A}_2$-SCAB $\mathcal{C}_{\mathcal{T}}$ as follows. Consider three vertices $v_1$, $v_2$, $v_3$: those will be the only vertices of $\mathcal{C}_{\mathcal{T}}$. Then, for each $x \in P$, put an edge $e_x$ between $v_1$ and $v_2$, an edge $e'_x$ between $v_2$ and $v_3$ and an edge $e''_x$ between $v_3$ and $v_1$. Finally, for each $(x,y,z) \in \mathcal{T}$, attach a triangle to the three edges $e_x$, $e'_y$ and $e''_z$. One readily checks that the definition of a triangle presentation ensures that the three rank~$2$ residues of $\mathcal{C}_{\mathcal{T}}$ are incidence graphs of the projective plane $\Pi$, and hence that $\mathcal{C}_{\mathcal{T}}$ is indeed an $\tilde{A}_2$-SCAB. Note that $\mathcal{C}_{\mathcal{T}}$ is not a simplicial complex in the usual sense as all simplices of dimension~$2$ have the same three vertices.

We then consider the \textbf{universal covering $\tilde{A}_2$-SCAB} $\tilde{\mathcal{C}}_{\mathcal{T}}$ of $\mathcal{C}_{\mathcal{T}}$, as defined in~\cite{Kantor}*{Definition~B.3.3, Proposition~B.3.4}. This universal covering $\tilde{\mathcal{C}}_{\mathcal{T}}$ is a simply connected simplicial complex of dimension~$2$ whose simplicial spheres of radius $1$ are isomorphic to the incidence graph of $\Pi$, so it is an $\tilde{A}_2$-building (see~\cite{Kantor}*{Theorem~B.3.8} for a more rigorous proof of this fact). We therefore set $\Delta_{\mathcal{T}} := \tilde{\mathcal{C}}_{\mathcal{T}}$. Moreover, because of (2) in Definition~\ref{definition:trianglepresentation}, there is an automorphism $\alpha \in \Aut(\mathcal{C}_\mathcal{T})$ sending $e_x$ to $e'_x$, $e'_x$ to $e''_x$ and $e''_x$ to $e_x$ for each $x \in P$. In other words, there is a natural action of the group $\mathbf{C}_3$ of order~$3$ on $\mathcal{C}_\mathcal{T}$. This automorphism group $\mathbf{C}_3$ then lifts to an automorphism group $\tilde{\mathbf{C}_3}$ of $\Delta_{\mathcal{T}}$ (see~\cite{Kantor}*{Corollary~B.3.7}), and $\tilde{\mathbf{C}_3}$ acts simply transitively on the set of vertices of $\Delta_{\mathcal{T}}$ (and by type-rotating automorphisms). The group $\Gamma_{\mathcal{T}}$ can thus be taken to be $\tilde{\mathbf{C}_3}$. The presentation of $\Gamma_{\mathcal{T}}$ given in Theorem~\ref{theorem:reconstruct} can finally be found by applying Theorem~\ref{theorem:presentation} to $\Gamma_\mathcal{T} \leq \Aut(\Delta_{\mathcal{T}})$.

\section{The strategy}\label{section:strategy}

A way to construct an $\tilde{A}_2$-building with non-Desarguesian residues and admitting a lattice is, in view of Theorem~\ref{theorem:reconstruct}, to consider a non-Desarguesian projective plane $\Pi$ and to find a point-line correspondence in $\Pi$ and a triangle presentation compatible with it. The smallest non-Desarguesian projective planes are the Hughes plane of order~$9$, the Hall plane of order~$9$ and the dual of the Hall plane. The Hughes plane is self-dual, so there exist some natural point-line correspondences in it: the correlations (also called dualities). For this reason, we decided to work on the Hughes plane of order $9$. It will appear later that the correlations do not actually admit a triangle presentation, but they will still be helpful in our search for a suitable point-line correspondence.

For each Desarguesian projective plane, Cartwright-Mantero-Steger-Zappa gave in \cite{CMSZ}*{\S4} an explicit formula for one point-line correspondence admitting a triangle presentation. Of course, they use the finite field from which the projective plane is constructed, and it is not clear how to find a similar formula for a particular non-Desarguesian projective plane.

Since we are searching for purely combinatorial objects, the use of a computer could be considered. In \cite{CMSZ2}, the authors used a computer to find all triangle presentations in the projective planes of order~$2$ and~$3$. The number of points in these projective planes being not too large (i.e. $7$ and $13$), they could do a \textit{brute-force} computation. However, already for order~$3$ they needed to use some symmetries of the problem so as to reduce the search space. Even if computers are now more powerful than in the 1990s, such a method would still be far too slow for a projective plane of order~$9$.

The key point is that we are not searching for all triangle presentations in the Hughes plane: we only want to find one. In this section, we describe our strategy in order to do~so.

\subsection{The graph associated to a point-line correspondence}\label{subsection:graph}

In the context of triangle presentations, it is natural to associate a particular graph to each point-line correspondence $\lambda\colon P \to L$ of a projective plane $\Pi$.

\begin{definition}
Let $\lambda \colon P \to L$ be a point-line correspondence in a projective plane $\Pi$. The \textbf{graph $G_\lambda$ associated to $\lambda$} is the directed graph with vertex set $V(G_\lambda) := P$ and edge set $E(G_\lambda) := \{(x,y) \in P^2 \mid y \in \lambda(x)\}$.
\end{definition}


For $\lambda$, admitting a triangle presentation can now be rephrased as a condition on its associated graph $G_\lambda$. In order to state this reformulation, we first define what we will call a \textit{triangle} in a directed graph.

\begin{definition}
Let $G$ be a directed graph. A set $\{e_1,e_2,e_3\}$ of edges in $G$ such that the destination vertex of $e_1$ (resp. $e_2$ and $e_3$) is the origin vertex of $e_2$ (resp. $e_3$ and $e_1$) is called a \textbf{triangle}. If two of the three edges $e_1$, $e_2$ and $e_3$ are equal, then they are all equal. In this case, the triangle contains only one edge and is also called a \textbf{loop}.
\end{definition}

The next definition will also be convenient.

\begin{definition}
Let $\lambda \colon P \to L$ be a point-line correspondence in a projective plane $\Pi$. A triple $(x,y,z)\in P^3$ is called \textbf{$\lambda$-admissible} if $y \in \lambda(x)$, $z \in \lambda(y)$ and $x \in \lambda(z)$.
\end{definition}

By definition, a triangle presentation compatible with $\lambda$ only contains $\lambda$-admissible triples. Thanks to these definitions, there is now an obvious bijection between triangles of $G_\lambda$ and (triples of) $\lambda$-admissible triples. Indeed, for $x,y,z \in P$, $(x,y,z)$ is $\lambda$-admissible if and only if there is a triangle $\{e_1,e_2,e_3\}$ in $G_\lambda$ with $x$, $y$ and $z$ being the origins of $e_1$, $e_2$ and $e_3$ respectively. Note that the triangle $\{e_1,e_2,e_3\}$ then corresponds to the three $\lambda$-admissible triples $(x,y,z)$, $(y,z,x)$ and $(z,x,y)$ (which are equal when $x=y=z$, i.e. when $e_1=e_2=e_3$ or equivalently when the triangle is a loop).

This observation directly gives us the next result.

\begin{lemma}\label{lemma:partition}
Let $\lambda \colon P \to L$ be a point-line correspondence in a projective plane $\Pi$. There exists a triangle presentation compatible with $\lambda$ if and only if there exists a partition of the set of edges $E(G_\lambda)$ of $G_\lambda$ into triangles.
\end{lemma}

\begin{proof}
Via the above bijection, a partition of $E(G_\lambda)$ into triangles exactly corresponds to a triangle presentation compatible with $\lambda$.
\end{proof}

\subsection{The score of a point-line correspondence}

Most point-line correspondences $\lambda$ in a projective plane do not admit a triangle presentation, i.e. the set of edges $E(G_\lambda)$ of the graph $G_\lambda$ can generally not be partitioned into triangles. We would still like to measure if a correspondence $\lambda$ is ``far from admitting" a triangle presentation or not. We therefore introduce the notion of a \textit{triangle partial presentation} compatible with $\lambda$.

\begin{definition}\label{definition:trianglepartialpresentation}
Let $\lambda \colon P \to L$ be a point-line correspondence in a projective plane $\Pi$. A subset $\mathcal{T} \subseteq P^3$ is called a \textbf{triangle partial presentation} compatible with $\lambda$ if the two following conditions hold:
\begin{enumerate}[(1)]
\item For all $x, y \in P$, if there exists $z \in P$ such that $(x,y,z) \in \mathcal{T}$ then $y \in \lambda(x)$ and $z$ is unique.
\item If $(x,y,z) \in \mathcal{T}$, then $(y,z,x) \in \mathcal{T}$.
\end{enumerate}
\end{definition}

We directly have the following.

\begin{lemma}\label{lemma:maxscore}
Let $\lambda \colon P \to L$ be a point-line correspondence in a projective plane $\Pi$ of order~$q$. A subset $\mathcal{T} \subseteq P^3$ is a triangle presentation compatible with $\lambda$ if and only if it is a triangle partial presentation compatible with $\lambda$ and $|\mathcal{T}| = (q+1)(q^2+q+1)$.
\end{lemma}

\begin{proof}
This is clear from the definitions, since there are exactly $(q+1)(q^2+q+1)$ pairs $(x,y) \in P^2$ with $y \in \lambda(x)$.
\end{proof}

We now define the \textit{score} of a point-line correspondence as follows.

\begin{definition}
Let $\lambda \colon P \to L$ be a point-line correspondence in a projective plane $\Pi$ of order~$q$. The \textbf{score} $S(\lambda)$ of $\lambda$ is the greatest possible size of a triangle partial presentation compatible with $\lambda$. 
\end{definition}

Thanks to the bijection between triangles of $G_\lambda$ and (triples of) $\lambda$-admissible triples (see \S\ref{subsection:graph}), we can restate this definition in the following terms.

\begin{definition}
Let $\lambda \colon P \to L$ be a point-line correspondence in a projective plane $\Pi$ of order~$q$. The \textbf{score} $S(\lambda)$ of $\lambda$ is the maximal number of edges of $G_\lambda$ that can be covered with disjoint triangles.
\end{definition}

A point-line correspondence then admits a triangle presentation if and only if its score reaches the maximal theoretical value $(q+1)(q^2+q+1)$.

\begin{lemma}\label{lemma:maxscore2}
Let $\lambda \colon P \to L$ be a point-line correspondence in a projective plane $\Pi$ of order~$q$. There exists a triangle presentation compatible with $\lambda$ if and only if $S(\lambda) = (q+1)(q^2+q+1)$.
\end{lemma}

\begin{proof}
This follows from Lemma~\ref{lemma:maxscore}.
\end{proof}

\subsection{Scores of correlations}

When $\lambda \colon P \to L, \ L \to P$ is a \textbf{correlation} of a (self-dual) projective plane $\Pi$ of order~$q$, i.e. a map such that $\lambda(p) \ni \lambda(\ell)$ if and only if $p \in \ell$, there is an explicit formula for the score of the point-line correspondence $\lambda \colon P \to L$.

\begin{proposition}\label{proposition:formula}
Let $\lambda \colon P \to L, \ L \to P$ be a correlation in a projective plane $\Pi$ of order~$q$. Let $a(\lambda)$ be the number of points $p \in P$ such that $\lambda^3(p) \ni p$ and let $b(\lambda)$ be the number of points $p \in P$ such that $\lambda^3(p) \ni p$ and $\lambda^6(p) = p$. Then
\[S(\lambda) = (q+1)(q^2+q+1) - (2q-3)\cdot a(\lambda) - b(\lambda).\]
\end{proposition}

\begin{proof}
For fixed $x, y \in P$ with $y \in \lambda(x)$ (i.e. $(x,y)$ is an edge of $G_\lambda$), a point $z \in P$ is such that $(x,y,z)$ is $\lambda$-admissible if and only if $z \in \lambda(y) \cap \lambda^{-1}(x)$. We call the edge $(x,y)$ \textbf{unpopular} if $\lambda(y) \neq \lambda^{-1}(x)$ and \textbf{popular} if $\lambda(y) = \lambda^{-1}(x)$. This means that an unpopular edge of $G_\lambda$ is contained in exactly one triangle while a popular edge is contained in exactly $(q+1)$ triangles.

\begin{enumerate}[(i)]
\item There are exactly $a(\lambda)$ popular edges in $G_\lambda$.\\ \textit{Proof:} By definition, $(x,y)$ is popular if $y = \lambda^{-2}(x)$, so a vertex $x \in P$ is the origin of a (unique) popular edge if and only if $\lambda^{-2}(x) \in \lambda(x)$, i.e. $x \in \lambda^3(x)$. There are exactly $a(\lambda)$ such $x$ and hence $a(\lambda)$ popular edges.
\item There are exactly $(q+1)(q^2+q+1) + q \cdot a(\lambda)$ $\lambda$-admissible triples.\\ \textit{Proof:} By (i), there are $(q+1)(q^2+q+1) - a(\lambda)$ unpopular edges and $a(\lambda)$ popular edges in $G_\lambda$. As each unpopular edge (resp. popular edge) is the beginning of one (resp. $(q+1)$) $\lambda$-admissible triple(s), we get
\[[(q+1)(q^2+q+1)-a(\lambda)]\cdot 1 + a(\lambda) \cdot (q+1)\]
$\lambda$-admissible triples.
\item There are exactly $(q+1) \cdot a(\lambda)$ $\lambda$-admissible triples $(x,y,z)$ with $(x,y)$ popular (resp. $(y,z)$ popular, $(z,x)$ popular). \\ \textit{Proof:} There are $a(\lambda)$ popular edges by (i), each one being the beginning of $(q+1)$ $\lambda$-admissible triples.
\item There are exactly $a(\lambda)$ $\lambda$-admissible triples $(x,y,z)$ with $(x,y)$ and $(y,z)$ popular (resp. $(y,z)$ and $(z, x)$ popular, $(z,x)$ and $(x,y)$ popular). \\ \textit{Proof:} If $(x,y,z)$ is $\lambda$-admissible with $(x,y)$ and $(y,z)$ popular, then $y = \lambda^{-2}(x)$, $z = \lambda^{-2}(y)$ and $x \in \lambda^3(x)$. Moreover, these conditions are sufficient to be $\lambda$-admissible with $(x,y)$ and $(y,z)$ popular. Since there are $a(\lambda)$ points $x$ such that $x \in \lambda^3(x)$, there are exactly $a(\lambda)$ such triples.
\item There are exactly $b(\lambda)$ $\lambda$-admissible triples $(x,y,z)$ with $(x,y)$, $(y,z)$ and $(z,x)$ popular. \\ \textit{Proof:} Such triples satisfy $x \in \lambda^3(x)$, $y = \lambda^{-2}(x)$, $z = \lambda^{-2}(y)$ and $x = \lambda^{-2}(z)$, so in particular $x = \lambda^6(x)$. Moreover, if $x \in \lambda^3(x)$ and $x = \lambda^6(x)$, then $(x, \lambda^{-2}(x), \lambda^{-4}(x))$ is $\lambda$-admissible with three popular edges, so there are exactly $b(\lambda)$ such triples.
\item There are exactly $(q+1)(q^2+q+1) - 2q \cdot a(\lambda) - b(\lambda)$ $\lambda$-admissible triples $(x,y,z)$ with $(x,y)$, $(y,z)$ and $(z,x)$ unpopular. \\ \textit{Proof:} By the inclusion-exclusion principle, the number of such triples is
\[[(q+1)(q^2+q+1) + q \cdot a(\lambda)] - 3(q+1)\cdot a(\lambda) + 3 \cdot a(\lambda) - b(\lambda).\]
\end{enumerate}
We now prove that $S(\lambda) \leq (q+1)(q^2+q+1) - (2q-3) \cdot a(\lambda) - b(\lambda)$. Let $\mathcal{T}$ be a triangle partial presentation with $|\mathcal{T}| = S(\lambda)$, i.e. a set of disjoint triangles of $G_\lambda$ covering $S(\lambda)$ edges. By maximality, all $\lambda$-admissible triples (i.e. triangles) $(x,y,z)$ with $(x,y)$, $(y,z)$ and $(z,x)$ unpopular are in $\mathcal{T}$ (because each of these $3$ edges is only covered by this particular triangle). By (vi), this means we already have $(q+1)(q^2+q+1) - 2q \cdot a(\lambda) - b(\lambda)$ triples in $\mathcal{T}$. The other triangles in $\mathcal{T}$ all contain at least one popular edge. There are $a(\lambda)$ popular edges (by (i)), so we obtain
\[S(\lambda) \leq (q+1)(q^2+q+1) - 2q \cdot a(\lambda) - b(\lambda) + 3\cdot a(\lambda).\]

Let us now show that $S(\lambda) \geq (q+1)(q^2+q+1) - (2q-3) \cdot a(\lambda) - b(\lambda)$, by covering that number of edges of $G_\lambda$ with disjoint triangles. We first cover exactly $(q+1)(q^2+q+1) - 2q\cdot a(\lambda)-b(\lambda)$ edges of $G_\lambda$ thanks to the triangles only containing unpopular edges. By definition of an unpopular edge, these triangles are all disjoint. Now, for each popular edge $(x,y)$, there are $(q+1)$ values of $z$ such that $(x,y,z)$ is $\lambda$-admissible. Among these $(q+1) \geq 3$ values of $z$, choose $z_0$ different from $\lambda^{-2}(y)$ and $\lambda^2(x)$. In this way, $(y, z_0)$ and $(z_0, x)$ are unpopular. We then add the triangle $(x,y,z_0)$ to our covering. This triangle is not a loop since $(x,y)$ is popular and $(y,z)$ is unpopular, so it covers three new edges. Doing so for each popular edge $(x,y)$, we cover $3 \cdot a(\lambda)$ new edges and get
\[ S(\lambda) \geq (q+1)(q^2+q+1) - 2q \cdot a(\lambda) - b(\lambda) + 3\cdot a(\lambda). \qedhere\]
\end{proof}

It follows from Proposition~\ref{proposition:formula} that a correlation $\lambda$ admits a triangle presentation if and only if $\lambda^3$ sends no point to an adjacent line. However, the following nice result of Devillers, Parkinson and Van Maldeghem shows that this never happens.

\begin{theorem}[Devillers--Parkinson--Van Maldeghem]\label{theorem:absolutepoints}
Let $\lambda \colon P \to L, \ L \to P$ be a correlation in a finite projective plane $\Pi$. Then there exists $p \in P$ such that $p \in \lambda(p)$.
\end{theorem}

\begin{proof}
See~\cite{duality}*{Proposition~5.4}. The case of polarities (i.e. correlations which are involutions) goes back to \cite{Baer}.
\end{proof}

\begin{corollary}\label{corollary:never}
Let $\lambda \colon P \to L, \ L \to P$ be a correlation in a finite projective plane $\Pi$. Then there is no triangle presentation compatible with $\lambda$.
\end{corollary}

\begin{proof}
Applying Theorem~\ref{theorem:absolutepoints} to the correlation $\lambda^3$, we get $a(\lambda) > 0$ and hence $S(\lambda) < (q+1)(q^2+q+1)$ by Proposition~\ref{proposition:formula}. The conclusion then follows from Lemma~\ref{lemma:maxscore2}.
\end{proof}

\begin{remark}
In the semifield plane of order $16$ and with kernel $\mathrm{GF}(4)$, we could observe a correlation $\lambda$ such that $a(\lambda) = b(\lambda) = 1$. This means that there is exactly one point $p$ of the plane such that $p \in \lambda^3(p)$. The score of this correlation $\lambda$ is thus $S(\lambda) = 4611$, the maximal theoretical score being $(16+1)(16^2+16+1) = 4641$.
\end{remark}

\subsection{Estimated score for a general point-line correspondence}

It does not seem possible to get a general formula for the score of all point-line correspondences. One can however obtain (good) lower bounds for the score, simply by trying to cover the most possible edges of $G_\lambda$ with triangles. There are different algorithms that could be used. Our principal goal being to know whether $E(G_\lambda)$ admits a partition into triangles, we should design an algorithm that will find such a partition when it exists. The idea is simple: if an edge of $G_\lambda$ is not yet covered and if there is only one triangle containing this edge and disjoint from the already chosen ones, then this triangle must be part of the (possible) partition. Our algorithm to cover as many edges as we can in $G_\lambda$ is thus the following:
\begin{enumerate}
\item[] While there exists $e \in E(G_\lambda)$ such that there is a unique triangle $t$ in $G_\lambda$ containing $e$, choose this triangle $t$, remove the edge(s) of $t$ from $G_\lambda$ and start again this procedure. If, at the end, there is no more triangles in $G_\lambda$, then we say that the score-algorithm \textbf{succeeds} and that the \textbf{estimated score} $s(\lambda)$ of $\lambda$ is the number of edges that are covered by the chosen triangles. Otherwise, there still are triangles in $G_\lambda$ but all edges are contained in $0$ or at least $2$ triangles. In this case, we say that the score-algorithm \textbf{fails}. For a pseudo-code, see~Algorithm~\ref{algorithm:score}.
\end{enumerate}

One should note that the value of $s(\lambda)$ (and whether the score-algorithm succeeds or not) may depend on the choice made for $e \in E(G_\lambda)$ at each step. We will still talk about \textit{the} estimated score $s(\lambda)$ of $\lambda$, assuming that an order is fixed once and for all on the set $E(G_\lambda)$ for each $\lambda$.

\begin{algorithm}[t!]
$\textit{score} \gets 0$\;
$\textit{edgesInOneTriangle} \gets true$\;
 \While{$\textit{edgesInOneTriangle} = true$}{
  $\textit{edgesInOneTriangle} \gets false$\;
  \For{$e$ in $E(G_\lambda)$}{
   \If{$e$ is contained in exactly one triangle $t$ of $G_\lambda$}{
    $\textit{edgesInOneTriangle} \gets true$\;
    remove the edge(s) of $t$ from $E(G_\lambda)$\;
    \eIf{$t$ is a loop}{
     $\textit{score} \gets \textit{score}+1$\;
    }{
     $\textit{score} \gets \textit{score}+3$\;
    }
   }
  }  
 }
 \eIf{there still are triangles in $G_\lambda$}{
   \Return{FAIL}
   }{
   \Return{score}
  }
 \caption{Computing the estimated score $s(\lambda)$ of $\lambda$}\label{algorithm:score}
\end{algorithm}

\begin{lemma}
Let $\lambda \colon P \to L$ be a point-line correspondence in a projective plane $\Pi$ of order~$q$. Assume that the score-algorithm succeeds. Then $s(\lambda) \leq S(\lambda)$ and, if $S(\lambda) = (q+1)(q^2+q+1)$, then $s(\lambda) = (q+1)(q^2+q+1)$.
\end{lemma}

\begin{proof}
This follows from the discussion in the description of the algorithm.
\end{proof}

When the score-algorithm fails, it cannot conclude whether there exists a partition of $E(G_\lambda)$ into triangles. Actually, we never encountered a point-line correspondence for which the algorithm fails for the Hughes plane of order~$9$. We therefore did not need to treat this particular case. Note however that, for a Desarguesian plane, we are aware of some point-line correspondences for which the algorithm fails and which indeed admit a triangle presentation, so this case should not in general be forgotten.

\subsection[Scores in the Hughes plane of order 9]{Scores in the Hughes plane of order~$9$}\label{subsection:Hughes}

By Corollary~\ref{corollary:never}, we know that a correlation never reaches the score of $(q+1)(q^2+q+1)$. A naive approach to find a point-line correspondence of the Hughes plane of order~$9$ with a score of $(9+1)(9^2+9+1) = 910$ is to simply evaluate $s(\lambda)$ for a lot of random correspondences $\lambda$ and to cross one's fingers. This idea is however not successful at all. Indeed, we computed the estimated score of $100000$ random point-line correspondences and got, on average, an estimated score of $486.6$ (with a standard deviation of $17.3$). The best estimated score we could observe was only $561$, very far from $910$.

Compared with these pretty low values, the formula given by Proposition~\ref{proposition:formula} for correlations seems to give better scores. In the Hughes plane of order~$9$, there are $33696$ correlations. Their scores, computed thanks to Proposition~\ref{proposition:formula}, are given in Table~\ref{table:scores}. Note that, as soon as two correlations $\lambda$ and $\lambda'$ are \textit{conjugate} (in the sense that $\lambda = \alpha \lambda' \alpha^{-1}$ for some automorphism $\alpha$ of the plane), we have $a(\lambda) = a(\lambda')$, $b(\lambda) = b(\lambda')$ and $S(\lambda) = S(\lambda')$. (Actually, $G_\lambda$ and $G_{\lambda'}$ are isomorphic.)

\begin{table}[t!]
\centering
\begin{tabular}{|c|c|c|c|c|}
\hline
$\#$ of concerned $\lambda$ & $a(\lambda)$ & $b(\lambda)$ & $S(\lambda)$ & $s(\lambda)$ (mean)\\
\hline
6318 & 4 & 4 & 846 & 846.00\\
4212 & 10 & 2 & 758 & 757.97\\
6318 & 10 & 10 & 750 & 750.00\\
4212 & 16 & 0 & 670 & 669.92\\
6318 & 16 & 16 & 654 & 654.00\\
6318 & 22 & 22 & 558 & 558.00\\
\hline
\end{tabular}
\caption{Scores of the correlations of the Hughes plane of order~$9$.}
\label{table:scores}
\end{table}

We also computed the estimated scores of all these correlations: they are also given in Table~\ref{table:scores}. They show that, at least for correlations, the estimated score is almost always equal to the real score. As expected, correlations have higher estimated scores than random point-line correspondences: they reach $846$. This fact will be helpful for our final strategy to find a correspondence with score $910$, described in the next subsection.

\subsection{Improving a point-line correspondence}

In order to find a point-line correspondence with a score greater than what we already obtained, it is natural to try to slightly modify a point-line correspondence with a high score. The smallest change we can make is to swap the images of two points. The next lemma shows that the score function is somewhat continuous.

\begin{lemma}\label{lemma:continuous}
Let $\lambda \colon P \to L$ be a point-line correspondence in a projective plane $\Pi$ of order~$q$ and let $a, b \in P$. Define $\lambda_{a,b} \colon P \to L$ by $\lambda_{a,b}(x) := \lambda(x)$ for all $x \in P \setminus \{a,b\}$, $\lambda_{a,b}(a) := \lambda(b)$ and $\lambda_{a,b}(b) := \lambda(a)$. Then $|S(\lambda_{a,b}) - S(\lambda)| \leq 6(q+1)$.
\end{lemma}

\begin{proof}
The graph $G_{\lambda_{a,b}}$ can be obtained from $G_{\lambda}$ by deleting the edges having $a$ or $b$ as origin and replacing them by other edges. In total, $2(q+1)$ edges are deleted and $2(q+1)$ edges are added. Since a triangle contains at most $3$ edges, we directly deduce that $|S(\lambda_{a,b})-S(\lambda)| \leq 6(q+1)$.
\end{proof}

In Lemma~\ref{lemma:continuous}, it is even reasonable to think that $|S(\lambda_{a,b})-S(\lambda)|$ will often be much smaller than $6(q+1)$. In other words, the score should not vary too much when replacing $\lambda$ by $\lambda_{a,b}$, and we can in general hope to have $S(\lambda_{a,b}) > S(\lambda)$ for some $a, b \in P$.

\begin{algorithm}[t!]
$\lambda \gets $ some correlation of the Hughes plane\;
 \While{$s(\lambda) < 910$}{
  $\textit{visited[}\lambda\textit{]} \gets \text{true}$\;
  $\textit{bestA} \gets -1$; $\textit{bestB} \gets -1$\;
  $\textit{bestScore} \gets -1$\;
  \For{$a$ in $P$ \textbf{and} $b$ in $P$}{
   \If{$\textit{visited[}\lambda_{a,b}\textit{]} = \text{false}$ \textbf{and} $s(\lambda_{a,b}) > \textit{bestScore}$}{
    $\textit{bestScore} \gets s(\lambda_{a,b})$\;
    $\textit{bestA} \gets a$\;
    $\textit{bestB} \gets b$\;
   }
  }
  $\lambda \gets \lambda_{\textit{bestA}, \textit{bestB}}$\;
 }
 \Return{$\lambda$}\;
 \caption{Finding a point-line correspondence $\lambda$ with $s(\lambda) = 910$}\label{algorithm:910}
\end{algorithm}

Based on this observation, our idea is simple. Start with a correlation $\lambda$, whose score is known to be higher than for a random correspondence (see \S\ref{subsection:Hughes}). For all distinct $a, b \in P$, consider $\lambda_{a,b}$ (as defined above) and compute its estimated score $s(\lambda_{a,b})$. Then choose $\tilde{a}, \tilde{b} \in P$ such that $s(\lambda_{\tilde{a},\tilde{b}}) = \max \{ s(\lambda_{a,b}) \mid a, b \in P\}$. Now replace $\lambda$ by $\lambda_{\tilde{a},\tilde{b}}$ and start this procedure again! We just need to keep track of the correspondences we already tried so as to avoid being blocked in a local maximum of the score function. This idea is explained in Algorithm~\ref{algorithm:910}. If after some time the algorithm does not seem able to produce a score of $910$, then we stop it and start it again from another correlation.

This procedure is pretty slow: with our implementation, one step (i.e. computing $s(\lambda_{a,b})$ for all $a, b \in P$ so as to find $\tilde{a}$ and $\tilde{b}$) takes ${\sim} 1.25$ seconds. For this reason and because we could still not reach $910$, we have decided not to try all possible pairs $a,b \in P$. Instead, we can observe which points seem to be the \textit{worst}, where the \textbf{badness} of $p \in P$ is the number of edges containing $p$ in $G_\lambda$ which were not covered by a triangle in Algorithm~\ref{algorithm:score}. Then, it is natural to only try the pairs $a,b \in P$ where $a$ is one of the worst points (for instance the $5$ worst points) and $b$ is arbitrary. Obviously, with this change Algorithm~\ref{algorithm:910} does not visit the same correspondences as before, but it has the advantage that a step only takes ${\sim}0.13$ seconds.

After three weeks of slight changes in the algorithm (e.g. the definition of a \textit{bad} point, the number of worst points we consider, the condition under which we stop and start with another correlation, etc), the computer eventually shouted (at least wrote) victory. The starting correlation had a score equal to $750$, and the evolution of the estimated score until $910$ is shown in Figure~\ref{figure:graph910}. 

\begin{remark}
The last change we made to the algorithm before it could solve the problem was actually mistaken! Whereas we wanted to speed up the computation of the five worst points, we made an error in the implementation of that idea resulting in the fact that the five computed points were actually not the worst ones. This mistake still led us to the discovery of a (valid) point-line correspondence $\lambda$ with a score of $910$. The funny part of the story is that if we correct this implementation error and start the algorithm with the same correlation, then it misses the correspondence $\lambda$.
\end{remark}

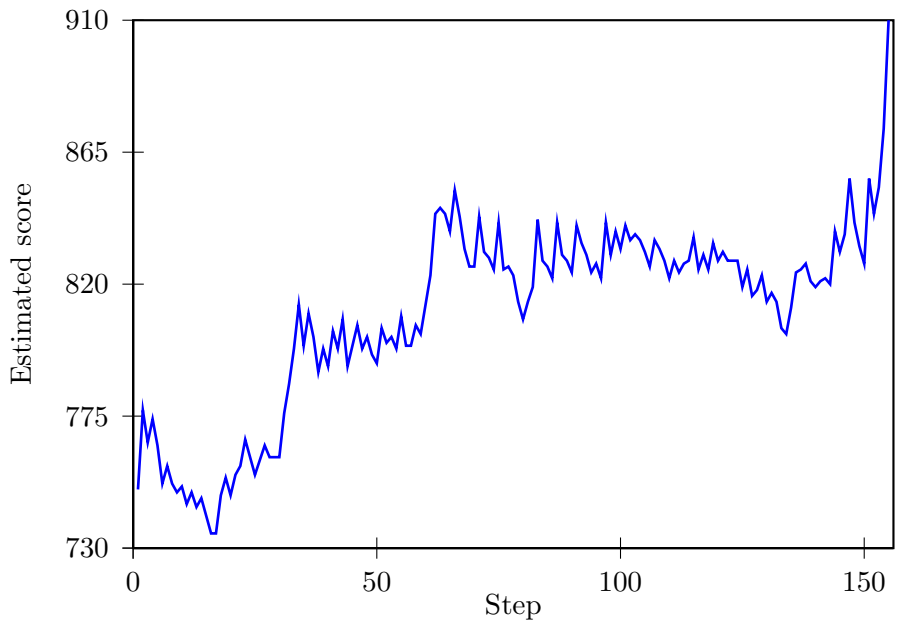
\begin{figure}[t!]
\centering
\begin{pspicture}(0,8.2)
\readdata[ignoreLines=1]{\mydata}{data.dat} 
\psset{xAxisLabel=Step,yAxisLabel=Estimated score,xAxisLabelPos={c,-20},yAxisLabelPos={-23,c}}
\psset{llx=-1cm,lly=-1cm}
\begin{psgraph}[axesstyle=frame,Dy=45,Dx=50,Ox=0,Oy=730](0,730)(0,730)(156,910){10cm}{7cm}
\listplot[linecolor=blue, linewidth=1pt,showpoints=false]{\mydata}
\end{psgraph}
\end{pspicture}
\caption{Evolution of the estimated score with Algorithm~\ref{algorithm:910}.}\label{figure:graph910}
\end{figure}

\section{The building and its lattice}\label{section:result}

In this section, we first give the description of the building and the lattice that we discovered (see Section~\ref{section:strategy} for the methods we used). We then give various properties of these objects (i.e. we prove (4), (5), (6) and (7) in Theorem~\ref{maintheorem}).

\subsection{Description}

The structure of the Hughes plane of order~$9$ is given in Table~\ref{table:incidence1} and comes from~\cite{projectiveplanes}. Points and lines are numbered from $0$ to $90$ (let us call them $p_0,\ldots,p_{90}$ and $\ell_0,\ldots,\ell_{90}$), and the $n^\text{th}$ row (with $0 \leq n \leq 90$) gives the indices of the $10$ points incident to $\ell_n$. The reader may be skeptical that the structure of incidence $\Pi$ defined by these point-line incidences is indeed the Hughes plane, but this is not so hard to verify by analyzing its properties. It is at least really easy to implement a program checking that $\Pi$ satisfies the axioms of a projective plane. Moreover, we can see by hand that $\Pi$ is not Desarguesian. For instance, consider the triangle $T_1$ with points $\{p_1,p_{10},p_{34}\}$ and the triangle $T_2$ with points $\{p_2,p_{11},p_{35}\}$. The lines of $T_1$ are $\{\ell_{10},\ell_{20},\ell_{11}\}$ and the lines of $T_2$ are $\{\ell_{27},\ell_{35},\ell_{19}\}$. Now the line passing through the points $\{p_1,p_2\}$ (resp. $\{p_{10},p_{11}\}$ and $\{p_{34},p_{35}\}$) is $\ell_0$ (resp. $\ell_1$ and $\ell_2$), and these three lines intersect in $p_0$. On the other hand, the common point of the lines $\{\ell_{10},\ell_{27}\}$ (resp. $\{\ell_{20},\ell_{35}\}$ and $\{\ell_{11},\ell_{19}\}$) is $p_{19}$ (resp. $p_{66}$ and $p_{42}$), and these three points do not lie on a common line (the line through $\{p_{19},p_{66}\}$ is $\ell_{48}$, which does not contain $p_{42}$). This shows that Desargues' theorem is not satisfied in $\Pi$. Moreover, our computations show that $\Pi$ is self-dual (since there are correlations), so it can only be the Hughes plane (see, for instance, \cite{Lam}).

Relative to this numbering of points and lines, the point-line correspondence $\lambda \colon P \to L$ which we have found is given in Table~\ref{table:lambda}. For the image of $p_{10x+y}$ by $\lambda$, one should look at the intersection of rows $x\_$ and $\_y$. The triangle presentation $\mathcal{T}$ compatible with $\lambda$ is then given in Table~\ref{table:triangle1}. In this table, the appearance of $(x,y,z)$ means that $(x,y,z)$, $(y,z,x)$ and $(z,x,y)$ all belong to $\mathcal{T}$. There are, in Table~\ref{table:triangle1}, $298$ triples $(x,y,z)$ with $x,y,z$ not all equal and $16$ triples $(x,x,x)$, which means that $\mathcal{T}$ contains $298 \cdot 3 + 16 = 910$ elements as required. While a computer helped to find $\mathcal{T}$, it can once again be checked by hand (or with a trivial program)  that $\mathcal{T}$ is indeed a triangle presentation compatible with $\lambda$. Indeed, one only needs to check that for each $(x,y,z) \in \mathcal{T}$, the line $\lambda(x)$ contains $y$ and there exists no $z' \neq z$ such that $(x,y,z') \in \mathcal{T}$. This suffices to show that $\mathcal{T}$ is a triangle presentation compatible with $\lambda$, since $|\mathcal{T}| = 910$.

It follows from Theorem~\ref{theorem:reconstruct} that the building $\Delta_\mathcal{T}$ and the group $\Gamma_{\mathcal{T}} \leq \Aut(\Delta_\mathcal{T})$ satisfy (1), (2) and (3) in Theorem~\ref{maintheorem}. In the next four subsections we prove (4), (5), (6) and~(7).

\subsection{Torsion in \texorpdfstring{$\Gamma_{\mathcal{T}}$}{Gamma}}

The group $\Gamma_{\mathcal{T}}$ has elements of order~$3$: when $(x,x,x) \in \mathcal{T}$ for some $x \in P$, we have the relation $a_x^3 = 1$ in the presentation of $\Gamma_{\mathcal{T}}$. However, the subgroup $\Gamma_{\mathcal{T}}^0$ of $\Gamma_{\mathcal{T}}$ consisting of the type-preserving automorphisms is torsion-free. Indeed, let $\gamma$ be a torsion element of $\Gamma_{\mathcal{T}}^0$, say of order~$n$. If $v_0$ is a fixed vertex of $\Delta_{\mathcal{T}}$, then $\gamma$ stabilizes the set $\{v_0, \gamma(v_0), \gamma^2(v_0), \ldots, \gamma^{n-1}(v_0)\}$. By \cite{BT}*{Proposition~3.2.4}, $\gamma$ must fix a point of $\Delta_{\mathcal{T}}$, i.e. it stabilizes a simplex of $\Delta_{\mathcal{T}}$. Since $\gamma$ preserves the types, it fixes this simplex pointwise and thus fixes its vertices. But $\Gamma_{\mathcal{T}}^0$ acts freely on the set of vertices of $\Delta_{\mathcal{T}}$, so $\gamma$ must be the identity element.

\subsection{A perfect subgroup of \texorpdfstring{$\Gamma_{\mathcal{T}}$}{Gamma}}

Clearly, $\Gamma_{\mathcal{T}}^0$ is a normal subgroup of index $3$ of $\Gamma_{\mathcal{T}}$. We find that $\Gamma_{\mathcal{T}}$ also has a subgroup of index $2$. Indeed, if we define $A \subset P$ by $A = \ell_{3} \cup \ell_{11} \cup \ell_{62} \cup \ell_{64} \cup \ell_{87}$, then one can check that for each $(x,y,z) \in \mathcal{T}$, either one or three of the points $x, y, z$ belong to $A$. Equivalently, either none or two of the points $x, y, z$ belong to $P \setminus A$. Hence, there is a well-defined group homomorphism $f \colon \Gamma_{\mathcal{T}} \to \mathbf{C}_2$ defined on the generators $\{a_x\}_{x \in P}$ by $f(a_x) := 0$ if $x \in A$ and $f(a_x) := 1$ if $x \not \in A$. The kernel $\ker(f)$ of $f$ is then a subgroup of index~$2$ of~$\Gamma_{\mathcal{T}}$.

The intersection $\Gamma_{\mathcal{T}}^0 \cap \ker(f)$ of these two subgroups is thus a normal subgroup of index $6$ of $\Gamma_{\mathcal{T}}$ (with $\bigslant{\Gamma_{\mathcal{T}}}{\Gamma_{\mathcal{T}}^0 \cap \ker(f)} \cong \mathbf{C}_2 \times \mathbf{C}_3$). We checked using the GAP system that $\Gamma_{\mathcal{T}}^0 \cap \ker(f)$ is a perfect group, so that $[\Gamma_{\mathcal{T}}, \Gamma_{\mathcal{T}}] = \Gamma_{\mathcal{T}}^0 \cap \ker(f)$.

\subsection{Partition of \texorpdfstring{$\Delta_{\mathcal{T}}$}{Delta} into sub-buildings}\label{subsection:partition}

A \textbf{Baer subplane} of a projective plane $\Pi$ is a proper projective subplane $\Pi_0$ of $\Pi$ with the property that every point of $\Pi$ is incident to at least one line of $\Pi_0$ and every line of $\Pi$ is incident to at least one point of $\Pi_0$. Let us take for $\Pi$ the Hughes plane of order~$9$. Then $\Pi$ has a (Desarguesian) Baer subplane $\Pi_0$ of order~$3$, which has the property that all automorphisms and all correlations of $\Pi$ preserve $\Pi_0$ (see \cite{Dembowski}*{5.4.1}). With respect to our numbering of the points and lines of the Hughes plane (see Table~\ref{table:incidence1}), the sets of points and lines of $\Pi_0$ are $$P_0 := \{p_n \mid n \in \{9, 17, 20, 33, 38, 42, 43, 46, 47, 56, 59, 64, 70\}\}$$
and $$L_0 := \{\ell_n \mid n \in \{3, 11, 22, 34, 46, 53, 62, 64, 70, 79, 84, 87, 89\}\}$$
(see the red-colored numbers in Table~\ref{table:incidence1}).

What is surprising is that our point-line correspondence $\lambda$ also preserves $\Pi_0$. This is indeed clear from Table~\ref{table:lambda}. Even better, if we call $\lambda_0$ the restriction of $\lambda$ to $P_0$, then the triangle presentation $\mathcal{T}$ can be restricted to a triangle presentation $\mathcal{T}_0$ compatible with $\lambda_0$. In other terms, for each $(x,y,z) \in \mathcal{T}$, if $x \in P_0$ and $y \in P_0$ then $z \in P_0$. This can also be simply observed by inspecting Table~\ref{table:triangle1}. The author does not know any theoretical reason why these properties are true (and whether they must be true for any point-line correspondence admitting a triangle presentation).

This observation has different consequences. First, we have a point-line correspondence $\lambda_0$ in the Desarguesian projective plane $\Pi_0$ of order~$3$, and a triangle presentation $\mathcal{T}_0$ compatible with it. Theorem~\ref{theorem:reconstruct} thus gives an $\tilde{A}_2$-building $\Delta_{\mathcal{T}_0}$ whose projective plane at each vertex is isomorphic to $\Pi_0$ and a group $\Gamma_{\mathcal{T}_0}$ acting simply transitively on $V(\Delta_{\mathcal{T}_0})$. The triangle presentations in the projective plane of order~$3$ have all been given by Cartwright-Mantero-Steger-Zappa in \cite{CMSZ2}, so $\mathcal{T}_0$ must be one of their list. It turns out that $\mathcal{T}_0$ is equivalent (as defined in \cite{CMSZ2}*{Section~2}) to their triangle presentation numbered $14.1$ (see \cite{CMSZ2}*{Appendix~B} ; one such equivalence takes the $p_n$, in the order listed in the definition of $P_0$, to $12, 2, 5, 0, 8, 11, 10, 3, 1, 9, 6, 4$ and $7$, respectively). In particular, this means by \cite{CMSZ2}*{Section~8} that $\Delta_{\mathcal{T}_0}$ is a non-linear building, i.e. is not the building of $\mathrm{PGL}(3,K)$ for some local field $K$.

The group $\Gamma_{\mathcal{T}_0}$ and the building $\Delta_{\mathcal{T}_0}$ also appear as subgroups and sub-buildings of $\Gamma_{\mathcal{T}}$ and $\Delta_{\mathcal{T}}$, respectively. With the notation of \S\ref{subsection:construction}, there is a clear embedding $e \colon \mathcal{C}_{\mathcal{T}_0} \hookrightarrow \mathcal{C}_{\mathcal{T}}$. Now $\Delta_{\mathcal{T}_0}$ and $\Delta_{\mathcal{T}}$ are the universal coverings of $\mathcal{C}_{\mathcal{T}_0}$ and $\mathcal{C}_{\mathcal{T}}$ respectively, so by fixing some vertices $v_0 \in V(\Delta_{\mathcal{T}_0})$ and $v \in V(\Delta_{\mathcal{T}})$ such that $p(v) = e(p_0(v_0))$ (where $p \colon \Delta_{\mathcal{T}} \to \mathcal{C}_{\mathcal{T}}$ and $p_0 \colon \Delta_{\mathcal{T}_0} \to \mathcal{C}_{\mathcal{T}_0}$ are the natural projections), we get an embedding $\tilde{e} \colon \Delta_{\mathcal{T}_0} \hookrightarrow \Delta_{\mathcal{T}}$ with $\tilde{e}(v_0) = v$. We can then see $\Gamma_{\mathcal{T}_0}$ as the subgroup of $\Gamma_{\mathcal{T}}$ such that $\Gamma_{\mathcal{T}_0}(v)$ is exactly the set of vertices of $\tilde{e}(\Delta_{\mathcal{T}_0})$. Moreover, for each $g \in \Gamma_{\mathcal{T}}$ the set $g \Gamma_{\mathcal{T}_0} (v) \subset V(\Delta_{\mathcal{T}})$ is also the $0$-skeleton of a building isomorphic to $\Delta_{\mathcal{T}_0}$. This means that the vertices of $\Delta_{\mathcal{T}}$ are partitioned into sub-buildings isomorphic to $\Delta_{\mathcal{T}_0}$ (where each sub-building corresponds to a left coset of $\Gamma_{\mathcal{T}_0}$ in $\Gamma_{\mathcal{T}}$).

One should note that these sub-buildings isomorphic to $\Delta_{\mathcal{T}_0}$ cover all the vertices of $\Delta_{\mathcal{T}}$, but this is not true for edges and chambers: some edges (and chambers) of $\Delta_{\mathcal{T}}$ do not belong to any of the sub-buildings.

\subsection{Automorphism group of \texorpdfstring{$\Delta_{\mathcal{T}}$}{Delta}}

The automorphism group $\Aut(\Delta_{\mathcal{T}})$ of $\Delta_{\mathcal{T}}$ contains $\Gamma_{\mathcal{T}}$, which acts simply transitively on the vertices of the building. In order to know whether $\Aut(\Delta_{\mathcal{T}})$ is substantially larger than $\Gamma_{\mathcal{T}}$, we should try to see what the stabilizer of a vertex in $\Aut(\Delta_{\mathcal{T}})$ looks like. This can be done by making use of the GAP system. I am very thankful to Tim Steger, who had done the same work for triangle presentations in the projective plane of order~$3$, and who gave me all his source codes and a great deal of advice.

Let $v$ be a vertex of $\Delta_{\mathcal{T}}$. In the next discussion, $X_{\mathcal{T}}$ will denote the sub-building of $\Delta_{\mathcal{T}}$ containing $v$ and isomorphic to $\Delta_{\mathcal{T}_0}$ (see \S\ref{subsection:partition}). We have the following facts.

\begin{enumerate}[(i)]
\item[(i)] Any automorphism of $\Delta_{\mathcal{T}}$ fixing $v$ must preserve the sub-building $X_{\mathcal{T}}$.\\ \textit{Explanation:} For a vertex $x$ contained in $X_{\mathcal{T}}$, there are $2 \cdot 91 = 182$ vertices adjacent to $x$ in $\Delta_{\mathcal{T}}$, and exactly $2 \cdot 13 = 26$ of them belong to $X_{\mathcal{T}}$. Those $26$ vertices are characterized by the fact that, in the local Hughes plane $\Pi$ associated to $x$, they correspond to the $13$ points and $13$ lines of the Baer subplane $\Pi_0$ of $\Pi$. Hence, if $\alpha \in \Aut(\Delta_{\mathcal{T}})$ is such that $\alpha(x) = y$ with $x, y$ belonging to $X_{\mathcal{T}}$, then $\alpha$ must send the $26$ neighbors of $x$ in $X_{\mathcal{T}}$ on the $26$ neighbors of $y$ in $X_{\mathcal{T}}$ because all automorphisms and correlations of $\Pi$ preserve $\Pi_0$. Starting with $x = v$, we obtain step by step that any automorphism of $\Delta_T$ fixing $v$ must stabilize $X_\mathcal{T}$.

\item[(ii)] There are exactly $16$ automorphisms of $X_{\mathcal{T}}$ stabilizing $v$. \\ \textit{Explanation:} This was previously done by Steger with the help of GAP.

\item[(iii)] For each $x \in V(\Delta_{\mathcal{T}})$, the only automorphism of the ball of radius $2$ centered at $x$ which pointwise stabilizes the ball of radius $1$ is the trivial automorphism. \\ \textit{Explanation:} This was proved with GAP.
\end{enumerate}

Point (iii) implies that the pointwise stabilizer of a ball of radius $1$ in $\Aut(\Delta_{\mathcal{T}})$ is trivial, and hence that an automorphism of $\Delta_{\mathcal{T}}$ is completely determined by  its action on the ball of radius $1$ centered at $v$. In particular, the stabilizer of $v$ in $\Aut(\Delta_{\mathcal{T}})$ is finite and $\Aut(\Delta_{\mathcal{T}})$ is discrete (for the topology of pointwise convergence).

\begin{enumerate}
\item[(iv)] There are exactly $6$ automorphisms of $\Pi$ that pointwise stabilize $\Pi_0$. \\ \textit{Explanation:} This can be checked with a computer, but one can also see~\cite{Luneburg}*{Corollary~5} or~\cite{Rosati} for a more theoretical approach.
\end{enumerate}

The four first points imply that there are at most $16 \cdot 6 = 96$ automorphisms of $\Delta_{\mathcal{T}}$ stabilizing $v$. Denote by $G_1$ the set of the $96$ automorphisms of the ball of radius~$1$ centered at $v$ which could maybe be extended to automorphisms of the whole building.

\begin{enumerate}
\item[(v)] Each automorphism in $G_1$ can be extended to an automorphism of the ball of radius $2$ centered at $v$. \\ \textit{Explanation:} This was proved with GAP.
\end{enumerate}

Now denote by $G_2$ the set of these extended automorphisms.

\begin{enumerate}
\item[(vi)] Each automorphism in $G_2$ can be extended to an automorphism of $\Delta_{\mathcal{T}}$. \\ \textit{Explanation:} This could be checked with a clever GAP program written by Steger.
\end{enumerate}

These steps actually gave us the explicit description of the $96$ automorphisms of $\Delta_{\mathcal{T}}$ fixing $v$. Six of them pointwise stabilize the sub-building $X_{\mathcal{T}}$. The file describing these automorphisms is pretty big so we do not append it to this text.

\newpage

\appendix

\section{The Hughes plane of order~9}\label{appendix:Hughes}

\setlength{\tabcolsep}{3.5pt}

\begin{table}[h!]
\footnotesize
\centering{
\begin{tabular}{c|cccccccccc}
0 & 0 & 1 & 2 & 3 & 4 & 5 & 6 & 7 & 8 & 9 \\
1 & 0 & 10 & 11 & 12 & 13 & 14 & 15 & 16 & 17 & 18 \\
2 & 0 & 19 & 34 & 35 & 36 & 37 & 38 & 39 & 40 & 41 \\
\textbx{3} & 0 & \textbx{20} & 27 & \textbx{42} & 55 & \textbx{56} & 57 & 58 & \textbx{59} & 60 \\
4 & 0 & 21 & 33 & 48 & 54 & 61 & 76 & 78 & 89 & 90 \\
5 & 0 & 22 & 30 & 43 & 49 & 63 & 68 & 72 & 79 & 80 \\
6 & 0 & 23 & 28 & 44 & 50 & 69 & 70 & 77 & 81 & 82 \\
7 & 0 & 24 & 29 & 45 & 51 & 64 & 73 & 74 & 83 & 84 \\
8 & 0 & 25 & 31 & 46 & 52 & 62 & 67 & 75 & 85 & 86 \\
9 & 0 & 26 & 32 & 47 & 53 & 65 & 66 & 71 & 87 & 88 \\
10 & 1 & 10 & 19 & 20 & 21 & 22 & 23 & 24 & 25 & 26 \\
\textbx{11} & 1 & 11 & 34 & \textbx{42} & \textbx{43} & 44 & 45 & \textbx{46} & \textbx{47} & 48 \\
12 & 1 & 12 & 28 & 35 & 55 & 61 & 62 & 63 & 64 & 65 \\
13 & 1 & 13 & 31 & 41 & 54 & 56 & 74 & 80 & 82 & 88 \\
14 & 1 & 14 & 33 & 36 & 50 & 58 & 68 & 73 & 85 & 87 \\
15 & 1 & 15 & 29 & 37 & 52 & 59 & 71 & 76 & 79 & 81 \\
16 & 1 & 16 & 27 & 38 & 51 & 66 & 72 & 77 & 86 & 89 \\
17 & 1 & 17 & 32 & 39 & 49 & 57 & 69 & 75 & 78 & 83 \\
18 & 1 & 18 & 30 & 40 & 53 & 60 & 67 & 70 & 84 & 90 \\
19 & 2 & 10 & 35 & 42 & 49 & 50 & 51 & 52 & 53 & 54 \\
20 & 3 & 10 & 29 & 34 & 56 & 61 & 66 & 67 & 68 & 69 \\
21 & 4 & 10 & 31 & 38 & 48 & 57 & 63 & 81 & 84 & 87 \\
\textbx{22} & 5 & 10 & \textbx{33} & 40 & \textbx{47} & \textbx{59} & \textbx{64} & 72 & 75 & 82 \\
23 & 6 & 10 & 28 & 41 & 43 & 58 & 71 & 83 & 86 & 90 \\
24 & 7 & 10 & 27 & 37 & 45 & 65 & 70 & 78 & 80 & 85 \\
25 & 8 & 10 & 30 & 39 & 46 & 55 & 73 & 76 & 77 & 88 \\
26 & 9 & 10 & 32 & 36 & 44 & 60 & 62 & 74 & 79 & 89 \\
27 & 2 & 11 & 19 & 27 & 28 & 29 & 30 & 31 & 32 & 33 \\
28 & 2 & 13 & 21 & 34 & 57 & 62 & 70 & 71 & 72 & 73 \\
29 & 2 & 14 & 22 & 37 & 48 & 60 & 64 & 69 & 86 & 88 \\
30 & 2 & 12 & 24 & 39 & 47 & 58 & 67 & 80 & 81 & 89 \\
31 & 2 & 18 & 20 & 41 & 45 & 61 & 75 & 77 & 79 & 87 \\
32 & 2 & 16 & 26 & 40 & 44 & 56 & 63 & 76 & 83 & 85 \\
33 & 2 & 15 & 25 & 36 & 43 & 55 & 66 & 78 & 82 & 84 \\
\textbx{34} & 2 & \textbx{17} & 23 & \textbx{38} & \textbx{46} & \textbx{59} & 65 & 68 & 74 & 90 \\
35 & 4 & 11 & 22 & 35 & 58 & 66 & 70 & 74 & 75 & 76 \\
36 & 5 & 11 & 21 & 39 & 50 & 56 & 65 & 79 & 84 & 86 \\
37 & 3 & 11 & 26 & 36 & 52 & 57 & 64 & 77 & 80 & 90 \\
38 & 6 & 11 & 20 & 40 & 51 & 62 & 68 & 78 & 81 & 88 \\
39 & 9 & 11 & 23 & 37 & 54 & 55 & 67 & 72 & 83 & 87 \\
40 & 8 & 11 & 24 & 38 & 49 & 60 & 61 & 71 & 82 & 85 \\
41 & 7 & 11 & 25 & 41 & 53 & 59 & 63 & 69 & 73 & 89 \\
42 & 5 & 14 & 19 & 42 & 63 & 67 & 71 & 74 & 77 & 78 \\
43 & 4 & 13 & 19 & 47 & 51 & 55 & 69 & 79 & 85 & 90 \\
44 & 3 & 16 & 19 & 43 & 54 & 60 & 65 & 73 & 75 & 81 \\
45 & 9 & 12 & 19 & 45 & 53 & 57 & 68 & 76 & 82 & 86 \\
\end{tabular}
\hspace{1cm}
\begin{tabular}{c|cccccccccc}
\textbx{46} & 6 & 15 & 19 & \textbx{46} & 49 & \textbx{56} & \textbx{64} & \textbx{70} & 87 & 89 \\
47 & 7 & 17 & 19 & 44 & 52 & 58 & 61 & 72 & 84 & 88 \\
48 & 8 & 18 & 19 & 48 & 50 & 59 & 62 & 66 & 80 & 83 \\
49 & 5 & 18 & 23 & 29 & 35 & 43 & 57 & 85 & 88 & 89 \\
50 & 3 & 13 & 24 & 30 & 35 & 44 & 59 & 78 & 86 & 87 \\
51 & 8 & 14 & 25 & 32 & 35 & 45 & 56 & 72 & 81 & 90 \\
52 & 7 & 15 & 21 & 31 & 35 & 47 & 60 & 68 & 77 & 83 \\
\textbx{53} & \textbx{9} & 16 & \textbx{20} & \textbx{33} & 35 & \textbx{46} & 69 & 71 & 80 & 84 \\
54 & 6 & 17 & 26 & 27 & 35 & 48 & 67 & 73 & 79 & 82 \\
55 & 4 & 14 & 20 & 30 & 34 & 52 & 65 & 82 & 83 & 89 \\
56 & 5 & 17 & 25 & 28 & 34 & 51 & 60 & 76 & 80 & 87 \\
57 & 6 & 12 & 22 & 32 & 34 & 54 & 59 & 77 & 84 & 85 \\
58 & 9 & 15 & 24 & 27 & 34 & 50 & 63 & 75 & 88 & 90 \\
59 & 8 & 16 & 23 & 31 & 34 & 53 & 58 & 64 & 78 & 79 \\
60 & 7 & 18 & 26 & 33 & 34 & 49 & 55 & 74 & 81 & 86 \\
61 & 4 & 15 & 23 & 32 & 40 & 42 & 61 & 73 & 80 & 86 \\
\textbx{62} & 3 & 12 & 25 & \textbx{33} & \textbx{38} & \textbx{42} & \textbx{70} & 79 & 83 & 88 \\
63 & 8 & 13 & 26 & 28 & 37 & 42 & 68 & 75 & 84 & 89 \\
\textbx{64} & \textbx{9} & \textbx{17} & 21 & 30 & 41 & \textbx{42} & \textbx{64} & 66 & 81 & 85 \\
65 & 7 & 16 & 22 & 29 & 39 & 42 & 62 & 82 & 87 & 90 \\
66 & 6 & 18 & 24 & 31 & 36 & 42 & 65 & 69 & 72 & 76 \\
67 & 4 & 12 & 26 & 29 & 41 & 46 & 50 & 60 & 72 & 78 \\
68 & 4 & 16 & 21 & 28 & 36 & 45 & 49 & 59 & 67 & 88 \\
69 & 4 & 18 & 25 & 27 & 39 & 44 & 54 & 64 & 68 & 71 \\
\textbx{70} & 4 & \textbx{17} & 24 & \textbx{33} & 37 & \textbx{43} & 53 & \textbx{56} & 62 & 77 \\
71 & 5 & 13 & 22 & 27 & 36 & 46 & 53 & 61 & 81 & 83 \\
72 & 5 & 12 & 20 & 31 & 37 & 44 & 49 & 66 & 73 & 90 \\
73 & 5 & 15 & 26 & 30 & 38 & 45 & 54 & 58 & 62 & 69 \\
74 & 5 & 16 & 24 & 32 & 41 & 48 & 52 & 55 & 68 & 70 \\
75 & 3 & 14 & 23 & 27 & 41 & 47 & 49 & 62 & 76 & 84 \\
76 & 3 & 18 & 21 & 32 & 37 & 46 & 51 & 58 & 63 & 82 \\
77 & 3 & 17 & 22 & 31 & 40 & 45 & 50 & 55 & 71 & 89 \\
78 & 3 & 15 & 20 & 28 & 39 & 48 & 53 & 72 & 74 & 85 \\
\textbx{79} & 7 & 13 & \textbx{20} & 32 & \textbx{38} & \textbx{43} & 50 & \textbx{64} & 67 & 76 \\
80 & 9 & 13 & 25 & 29 & 40 & 48 & 49 & 58 & 65 & 77 \\
81 & 6 & 13 & 23 & 33 & 39 & 45 & 52 & 60 & 63 & 66 \\
82 & 6 & 14 & 21 & 29 & 38 & 44 & 53 & 55 & 75 & 80 \\
83 & 7 & 14 & 24 & 28 & 40 & 46 & 54 & 57 & 66 & 79 \\
\textbx{84} & \textbx{9} & 14 & 26 & 31 & 39 & \textbx{43} & 51 & \textbx{59} & 61 & \textbx{70} \\
85 & 8 & 12 & 21 & 27 & 40 & 43 & 52 & 69 & 74 & 87 \\
86 & 7 & 12 & 23 & 30 & 36 & 48 & 51 & 56 & 71 & 75 \\
\textbx{87} & \textbx{9} & 18 & 22 & 28 & \textbx{38} & \textbx{47} & 52 & \textbx{56} & 73 & 78 \\
88 & 8 & 15 & 22 & 33 & 41 & 44 & 51 & 57 & 65 & 67 \\
\textbx{89} & 8 & \textbx{17} & \textbx{20} & 29 & 36 & \textbx{47} & 54 & 63 & \textbx{70} & 86 \\
90 & 6 & 16 & 25 & 30 & 37 & 47 & 50 & 57 & 61 & 74 \\
\end{tabular}
}
\caption{Incidence relation of the Hughes plane of order $9$, from \cite{projectiveplanes}.}
\label{table:incidence1}
\end{table}

\newpage

\section{The triangle presentation}
\begin{table}[h!]
\footnotesize
\centering
\begin{tabular}{c|cccccccccc}
$\lambda$ & $\_$0 & $\_$1 & $\_$2 & $\_$3 & $\_$4 & $\_$5 & $\_$6 & $\_$7 & $\_$8 & $\_$9\\
\hline
0$\_$ & 20 & 0 & 44 & 75 & 78 & 77 & 50 & 76 & 37 & \textbx{3} \\
1$\_$ & 54 & 39 & 30 & 8 & 88 & 68 & 18 & \textbx{34} & 65 & 57 \\
2$\_$ & \textbx{70} & 82 & 42 & 23 & 38 & 90 & 81 & 13 & 61 & 69 \\
3$\_$ & 73 & 4 & 83 & \textbx{22} & 58 & 28 & 59 & 55 & \textbx{64} & 60 \\
4$\_$ & 56 & 2 & \textbx{87} & \textbx{84} & 26 & 45 & \textbx{53} & \textbx{11} & 80 & 41 \\
5$\_$ & 25 & 14 & 63 & 72 & 7 & 32 & \textbx{62} & 86 & 51 & \textbx{46} \\
6$\_$ & 36 & 27 & 31 & 29 & \textbx{79} & 33 & 16 & 71 & 85 & 24 \\
7$\_$ & \textbx{89} & 35 & 17 & 19 & 5 & 47 & 67 & 10 & 66 & 43 \\
8$\_$ & 6 & 21 & 1 & 52 & 74 & 40 & 12 & 48 & 9 & 15 \\
9$\_$ & 49\\
\end{tabular}
\caption{Point-line correspondence $\lambda$.}
\label{table:lambda}
\end{table}

\def\arraystretch{0.95}
\setlength{\tabcolsep}{3.5pt}

\begin{table}[h!]
\footnotesize
\begin{tabular}{ccccccccc}
(0,3,41)   &
(0,10,82)   &
(0,29,54)   &
(0,34,9)   &
(0,56,88)   &
(0,61,31)   &
(0,66,1)   &
(0,67,13)   &
(0,68,74)   \\
(0,69,80)   &
(1,1,1)     &
(1,2,16)   &
(1,3,47)   &
(1,4,72)   &
(1,5,89)   &
(1,6,86)   &
(1,7,51)   &
(1,8,77)   \\
(1,9,27)   &
(2,3,62)   &
(2,19,12)   &
(2,43,61)   &
(2,54,73)   &
(2,60,65)   &
(2,65,55)   &
(2,73,35)   &
(2,75,17)   \\
(2,81,63)   &
(3,3,3)   &
(3,14,8)   &
(3,23,6)   &
(3,27,56)   &
(3,49,7)   &
(3,76,4)   &
(3,84,5)   &
(4,15,28)   \\
(4,20,37)   &
(4,28,15)   &
(4,39,81)   &
(4,48,29)   &
(4,53,20)   &
(4,74,79)   &
(4,85,71)   &
(5,17,90)   &
(5,22,67)   \\
(5,31,33)   &
(5,40,60)   &
(5,45,53)   &
(5,50,30)   &
(5,55,40)   &
(5,71,22)   &
(6,13,25)   &
(6,24,78)   &
(6,30,26)   \\
(6,35,21)   &
(6,44,10)   &
(6,59,19)   &
(6,78,24)   &
(6,87,59)   &
(7,18,39)   &
(7,21,75)   &
(7,32,57)   &
(7,37,83)   \\
(7,46,69)   &
(7,58,32)   &
(7,63,64)   &
(7,82,18)   &
(8,11,87)   &
(8,26,52)   &
(8,36,58)   &
(8,52,68)   &
(8,57,36)   \\
(8,64,50)   &
(8,80,70)   &
(8,90,85)   &
\textbx{(9,20,43)}   &
\textbx{(9,42,38)}   &
(9,55,44)   &
\textbx{(9,56,42)}   &
(9,57,48)   &
(9,58,45)   \\
\textbx{(9,59,46)}   &
(9,60,11)   &
(10,17,23)   &
(10,26,33)   &
(10,27,82)   &
(10,35,73)   &
(10,48,77)   &
(10,67,81)   &
(10,73,50)   \\
(10,79,69)   &
(11,11,11)   &
(11,23,71)   &
(11,37,82)   &
(11,54,24)   &
(11,55,85)   &
(11,67,61)   &
(11,72,49)   &
(11,83,47)   \\
(12,12,12)   &
(12,24,68)   &
(12,39,86)   &
(12,47,45)   &
(12,58,56)   &
(12,67,53)   &
(12,80,82)   &
(12,81,57)   &
(12,89,76)   \\
(13,31,48)   &
(13,46,35)   &
(13,52,26)   &
(13,62,79)   &
(13,67,27)   &
(13,75,52)   &
(13,85,82)   &
(13,86,64)   &
(14,15,21)   \\
(14,22,63)   &
(14,33,82)   &
(14,41,37)   &
(14,44,32)   &
(14,51,58)   &
(14,57,51)   &
(14,65,43)   &
(14,67,22)   &
(15,16,30)   \\
(15,36,34)   &
(15,45,82)   &
(15,49,59)   &
(15,59,89)   &
(15,67,83)   &
(15,88,65)   &
(16,18,82)   &
(16,40,25)   &
(16,53,66)   \\
(16,60,84)   &
(16,67,46)   &
(16,70,36)   &
(16,84,55)   &
(16,90,18)   &
\textbx{(17,17,38)}   &
\textbx{(17,46,20)}   &
\textbx{(17,59,70)}   &
(17,65,82)   \\
(17,68,40)   &
(17,74,72)   &
(18,22,42)   &
(18,29,39)   &
(18,42,78)   &
(18,62,87)   &
(18,87,62)   &
(18,90,29)   &
(19,22,77)   \\
(19,32,79)   &
(19,34,75)   &
(19,54,45)   &
(19,59,87)   &
(19,77,22)   &
(19,84,41)   &
(19,85,61)   &
(20,24,62)   &
\textbx{(20,33,64)}   \\
(20,53,37)   &
\textbx{(20,56,70)}   &
(20,62,77)   &
(20,77,24)   &
(21,21,38)   &
(21,29,68)   &
(21,44,60)   &
(21,53,31)   &
(21,55,83)   \\
(21,80,77)   &
(22,71,74)   &
(22,74,63)   &
(22,78,42)   &
(23,28,80)   &
(23,41,36)   &
(23,43,26)   &
(23,58,90)   &
(23,83,77)   \\
(23,86,28)   &
(23,90,57)   &
(24,40,34)   &
(24,51,85)   &
(24,81,84)   &
(24,88,32)   &
(25,25,25)   &
(25,30,58)   &
(25,37,65)   \\
(25,47,48)   &
(25,50,77)   &
(25,57,56)   &
(25,61,29)   &
(25,74,49)   &
(26,39,55)   &
(26,45,76)   &
(26,60,39)   &
(26,63,88)   \\
(26,66,77)   &
(27,31,61)   &
(27,41,34)   &
(27,54,29)   &
(27,74,68)   &
(27,80,69)   &
(27,88,66)   &
(28,32,40)   &
(28,40,80)   \\
(28,42,52)   &
(28,61,32)   &
(28,73,42)   &
(28,86,61)   &
(29,44,89)   &
(29,64,76)   &
(29,71,70)   &
(30,30,38)   &
(30,45,57)   \\
(30,54,74)   &
(30,62,61)   &
(30,69,37)   &
(31,54,83)   &
(31,76,78)   &
(31,78,36)   &
(31,89,81)   &
(31,90,43)   &
(32,46,84)   \\
(32,54,64)   &
(32,66,72)   &
\textbx{(33,33,33)}   &
(33,40,51)   &
\textbx{(33,47,46)}   &
\textbx{(33,59,56)}   &
(33,72,39)   &
(33,75,61)   &
(34,34,34)   \\
(34,50,39)   &
(34,63,37)   &
(34,88,47)   &
(34,90,35)   &
(35,57,71)   &
(35,62,41)   &
(35,70,86)   &
(35,71,58)   &
(35,72,83)   \\
(36,53,44)   &
(36,64,67)   &
(36,78,65)   &
(36,79,51)   &
(37,52,89)   &
(37,89,52)   &
(38,41,41)   &
\textbx{(38,42,56)}   &
\textbx{(38,64,64)}   \\
(38,66,66)   &
(38,81,81)   &
(38,85,85)   &
(39,49,41)   &
(39,74,43)   &
(40,76,41)   &
(40,87,48)   &
\textbx{(42,47,47)}   &
(42,73,52)   \\
\textbx{(43,43,43)}   &
(43,51,68)   &
\textbx{(43,59,64)}   &
\textbx{(43,70,47)}   &
(44,44,44)   &
(44,62,75)   &
(44,74,80)   &
(44,79,47)   &
(45,45,45)   \\
(45,68,69)   &
(45,86,62)   &
\textbx{(46,46,46)}   &
(46,71,76)   &
(46,80,50)   &
(48,48,48)   &
(48,49,63)   &
(48,58,81)   &
(48,65,84)   \\
(49,53,73)   &
(49,69,85)   &
(49,73,53)   &
(49,89,59)   &
(50,55,76)   &
(50,73,51)   &
(50,76,60)   &
(50,88,87)   &
(51,73,54)   \\
(51,87,66)   &
(52,75,84)   &
(52,84,68)   &
(53,90,88)   &
(54,84,70)   &
(55,56,79)   &
(55,63,86)   &
(56,83,60)   &
(57,75,72)   \\
(58,72,75)   &
(60,79,85)   &
(60,86,63)   &
(63,69,70)   &
(65,66,86)   &
(65,78,69)   &
(66,89,71)   &
(68,87,83)   &
(69,78,72)   \\
\textbx{(70,70,70)}   &
(71,75,88)   &
(72,78,76)   &
(79,79,79)   &
(79,90,89)   &
(80,81,87)   &
(83,83,83)   &
(88,88,88)   \\
\end{tabular}
\caption{Triangle presentation $\mathcal{T}$ compatible with $\lambda$.}
\label{table:triangle1}
\end{table}



\end{document}